\documentclass[aip,graphicx]{revtex4-1}

\draft % marks overfull lines with a black rule on the right

\usepackage{geometry}                % See geometry.pdf to learn the layout options. There are lots.
\usepackage{hyperref,graphicx, tikz}
\usepackage{amssymb,amsmath,amsthm}%amsmath
\usepackage{epstopdf}
\usepackage{setspace}%, concmath}
\usepackage[T1]{fontenc}
\usepackage{mathptmx}

\usepackage{tikz}
\usetikzlibrary{calc,arrows,automata,angles}

\newtheorem{lemma}{Lemma}

\newtheorem{theorem}{Theorem}

\newtheorem{proposition}{Proposition}

\newcommand{\R}{\mathbb{R}}
\newcommand{\Z}{\mathbb{Z}}
\newcommand{\N}{\mathbb{N}}
\newcommand{\one}{\vec {\mathbf 1}}

\newcommand{\equilibrioalfa}[3]{
\begin{tikzpicture}
      [scale=#1,place/.style={circle,draw=black,thick,fill=black, inner sep=0pt,minimum size=#2mm}]

    \foreach \x in {0,1,2,3}{
      \node (\x) at (90*\x-#3:2) [place] {};
      \node (\x') at (90*\x+#3:2) [place] {};
    };
    \foreach \x in {0,1,2,3}{
    \draw (\x)--(\x');
    \node  at (90*\x:2.2)  {$c$};
    }
    \foreach \x/\y in {0/1,1/2,2/3,3/0}{
    \draw (\x')--(\y);
    \node  at (45+90*\x:1.8)  {$a$};
    }
    \foreach \x/\y in {0/1,1/2,2/3,3/0}{
    \draw (\x)--(\y);
        \draw (\x')--(\y');
    \node  at (75+90*\x:1.7)  {$b$};
        \node  at (20+90*\x:1.6)  {$b$};
    }
    \foreach \x/\y in {0/1,1/2,2/3,3/0}{
    \draw (\x)--(\y');
    \node  at (45+90*\x:1)  {$d$};
    }
    \foreach \x/\y in {0/1,1/2,2/3,3/0}{
    \draw[dotted] (0,0)--(\x);
    \draw[dotted] (0,0)--(\x');
    \node  at (90*\x:0.7)  {$\beta$}; 
    \node  at (45+90*\x:0.4)  {$\alpha$}; 
    }
   \foreach \x/\y in {0/1,1/2,2/3,3/0}{
   \draw[thick, gray] (90*\x+#3:0.25) arc (90*\x+#3:90*\x+90-2*#3:0.25);
   \draw[thick, gray] (90*\x-#3:0.5) arc (90*\x-#3:90*\x+#3:0.5);   
}           
    \end{tikzpicture}  
}

\newcommand{\Gabcdconcrete}[4]{
\begin{tikzpicture}
      [scale=1.15,place/.style={circle,draw=black,thick,fill=black, inner sep=0pt,minimum size=2mm}]

    \foreach \x in {0,1,2,3}{
      \node (\x) at (90*\x-10:2) [place] {};
      \node (\x') at (90*\x+10:2) [place] {};
    };
    \foreach \x in {0,1,2,3}{
    \draw (\x)--(\x');
    \node[scale=0.6]  at (90*\x:2.2)  {$#3$};
    }
    \foreach \x/\y in {0/1,1/2,2/3,3/0}{
    \draw (\x')--(\y);
    \node[scale=0.6]  at (45+90*\x:1.8)  {$#1$};
    }
    \foreach \x/\y in {0/1,1/2,2/3,3/0}{
    \draw (\x)--(\y);
        \draw (\x')--(\y');
    \node[scale=0.6]  at (75+90*\x:1.7)  {$#2$};
        \node[scale=0.6]  at (20+90*\x:1.6)  {$#2$};
    }
    \foreach \x/\y in {0/1,1/2,2/3,3/0}{
    \draw (\x)--(\y');
    \node[scale=0.6]  at (45+90*\x:1)  {$#4$};
    }
    \end{tikzpicture} 
}

\newcommand{\Gabcdk}[7]{
\begin{tikzpicture}
      [scale=#7,place/.style={circle,draw=black,thick,fill=black, inner sep=0pt,minimum size=1mm},gris/.style={circle,draw=gray,thick,fill=gray, inner sep=0pt,minimum size=1mm}]

    \foreach \x/\y/\z in {0/0/7,1/2/1,2/4/3,3/6/5}{
        \foreach \k in {0,...,#5}{
      \node (\x\k) at    ($ (90*\x-15:2) + (90*\x-15+#6*\k:0.2) $)  [place] {};
      \node (\x'\k) at    ($ (90*\x+15:2) + (90*\x+15+#6*\k:0.2) $)  [place] {};
    };};
        \foreach \x in {0,1,2,3}{
        \foreach \k in {0,...,#5}{
          \foreach \kk in {0,...,#5}{
		\draw[thick] (\x\k)--(\x\kk);
		\draw[thick] (\x'\k)--(\x'\kk);
};    };};
    \foreach \x in {0,1,2,3}{
	\foreach \k in {0,1,...,#5}{
	\foreach \c in {0,...,#3}{
		\draw[thick, color=gray] (\x\k)to      ($ (90*\x+15:2) + (90*\x+15+#6*\c+#6*\k:0.2)$);
    };};};
    \foreach \x/\y in {0/1,1/2,2/3,3/0}{
	\foreach \k in {0,...,#5}{
	 \foreach \d in {0,...,#4}{
			    \draw (\x\k)-- ($ (90*\y+15:2) + (90*\y+15+#6*\d+#6*\k:0.2) $);

    };};};
    \foreach \x/\y in {0/1,1/2,2/3,3/0}{
    \foreach \k in {0,...,#5}{
	\foreach \a in {0,...,#1}{
				    \draw[gray] (\x'\k)-- ($ (90*\y-15:2) + (90*\y-15+#6*\a+#6*\k:0.2) $);
    };};};

   \foreach \x/\y in {0/1,1/2,2/3,3/0}{
            \foreach \k in {0,...,#5}{
                        \foreach \b in {0,...,#2}{
    \draw[dotted] (\x'\k)-- ($ (90*\y+15:2) + (90*\y+15+#6*\b+#6*\k:0.2) $);
    };};};

   \foreach \x/\y/\z in {0/0/7,1/2/1,2/4/3,3/6/5}{
        \foreach \k in {0,...,#5}{
      \node (\x\k) at    ($ (90*\x-15:2) + (90*\x-15+#6*\k:0.2) $)  [gris] {};
      \node[scale=1] at    ($ (90*\x-15:2) + (90*\x-15+#6*\k:0.4) $)   {\text{\tiny $(\z,\k)$}}; 
      \node[scale=1] at    ($ (90*\x+15:2) + (90*\x+15+#6*\k:0.4) $)   {\text{\tiny $(\y,\k)$}}; 
      \node (\x'\k) at    ($ (90*\x+15:2) + (90*\x+15+#6*\k:0.2) $)  [gris] {};
    };};

    \end{tikzpicture}
}

\begin{document}

%\preprint{}

\title{From weighted to unweighted graphs in Synchronizing Graph Theory.} %Title of paper



\author{Eduardo A. Canale}
\email[]{eduardo.canale@gmail.com, canale@fing.edu.uy}
%\thanks{AIP-Chaos-v5}
\affiliation{Instituto de Matem\'atica y Estad\'istica, Universidad de la Rep\'ublica, Uruguay}


\date{\today}

\begin{abstract}
A way to associate unweighted graphs from weighted ones is presented, such that linear stable equilibria of the Kuramoto homogeneous model associated to both graphs coincide, i.e., equilibria of  the system $\dot\theta_i  = \sum_{j \sim i} \sin(\theta_{j}-\theta_j)$, where $i\sim j$ means vertices $i$ and $j$ are adjacent in the corresponding graph.
As a consequence, the existence of  linearly stable equilibrium  is proved to be NP-Hard as conjectured by R. Taylor in 2015 and  a new lower bound for the minimum degree that ensures synchronization is found.

\end{abstract}


\maketitle %\maketitle must follow title, authors, abstract and \pacs
\begin{quotation}
Oscillator synchronization is a behavior found  in many complex systems, so it is of paramount interest  to understand the conditions that ensure it.
One fundamental piece of this puzzle is  the topology of the  interconnection network.
The homogeneous Kuramoto model stands out as a good model to answer that question, since its dynamics depends only on that topology.
Many concepts from  graph theory have been imported in order to shed light on this topic, like  biconnectivity,  twin vertices   and minimum degree. This last  parameter has been proved to imply, when large enough, that almost all initial conditions go to  an in-phase sync. How large it should be  is a limit   called \emph{critical connectivity} and it is known to be between 0.6838 and 0.75.
The  lower bound of 0.6838, has been improved in the last ten years  by theoretical and computational effort from an initial value of 0.6809 to 0.6818, then to 0.6828, and finally  to 0.6838.  All  these improvements were achieved by means of the search of circulant graphs with  no trivial linearly stable equilibria.
However the last of these  improvements close the possibility of finding new ones from circulant graphs.
In this work, we improve the lower bound to 0.6875 by means of non-circulant graphs.
Moreover, we have considered a  general Kuramoto model, with different strengths  between the oscillators and developed a way to assign a   model with equal strengths  that share its linearly stable equilibria. This technique also allows us to prove that the computational complexity of classifying equilibria by linear analysis is NP-Hard.
Therefore the method has opened a gate between the more general world of weighted graphs to the more restricted world of unweighted ones, which gives hope for a whole set of  new results in the field.
\end{quotation}

\section{Introduction}

In 2012 a seminal work by Richard Taylor \cite{taylor2012there} did a big step toward the classification of synchronizing graphs proving the existence of a critical value  $\mu_c<1$, called \emph{critical connectivity}\cite{Kassabov}  such that  almost any orbit of the homogeneous Kuramoto model associated with a graph  with $n$  vertices and minimum degree at least $\mu(n-1)$, ensure the graph is \emph{synchronizing}, i.e.,  almost any orbit should go to an equilibrium of the form $(c,c,\ldots,c)$. 
In that work, the author gives an upper and a lower bound for $\mu_c$. 
Both bounds were later  remarkably improved.
The upper bound was improved seven years later  from 0.9395 to 0.7929  by Ling, Xu and Bandeira \cite{ling}, then from 0.7929 to  0.7889 by Lu and Steinerberger \cite{Lu} and finally from 0.7889 to  0.75 by  Kassabov, Strogatz and  Townsend \cite{Kassabov}, in a  time-slaps of two years. 
On the other hand, the lower bound was improved  in 2015 from 0.6809 to 0.6818 by Canale and Monz\'on\cite{canale2015exotic} by means of Kroenecker  product of circulant graphs with complete ones, then from   0.6918 to 0.6828 by Townsend, Stillman and  Strogatz\cite{townsend2020dense}  by the exhaustive analysis of graphs with at most 50 vertices and, finally, in an what we consider an amazing work, from 0.6828 to 0.6838    by Yoneda,  Tatsukawa, and  Teramaer \cite{Yoneda}, were they take in count all possible circulant graphs, transform the problem in a family of  integer programing problems, solve  each of them  analytically to  find an optimal solution.
In the present work we increase  the  lower bound from $0.6838$ to $0.6875$  by giving an explicit example of a family of non synchronizing graphs.
Our construction comes from a kind of perturbation of the tensor product of $C_4$ with the complete graph suggested by Kassabov et al.\cite{Kassabov}, but going outside the family of circulant graphs. The construction is based on a way to transform weighted graph into unweighted ones.
As a collateral result we close an open problem on the complexity of classifying stable equilibria left by Taylor~\cite{Taylor2015}, proving that this decision problem is NP-Hard.

\section{Preliminaries}

\subsection{Graph theory}

A  \emph{graph} $G$ consists of a set $VG$ of \emph{vertices},  some of them \emph{joined} by  \emph{edges} in a set $EG$.
If two vertices $v$ and $w$ are joined by an edge $e$, we  say they are \emph{adjacent} and we write $e=vw$  or simply $v \sim w$ if no doubt about $G$ could arise.
In this work, all graphs are \emph{simple}, i.e. there are no edge of the form $vv$ and  no two different edges join the same vertices. 
The graph is a \emph{weighted graph}  if it comes  together with a \emph{weight function} $w:EG \to \N $. If $w$ is constant equal to 1, then the graph can be considered unweighted.

The \emph{order} $|G|$ of $G$ is the cardinality $|VG|$ of its vertex set while 
we  call  \emph{weight } of $G$ to the sum $w(G)$ of all its weights, i.e. $w(G) = \sum_{e\in EG} w_e$.
We will denote by $G_v$  the set  of vertices adjacent with $v$ in $G$, thus,   $w \in G_v$ iff $v \sim w$. 
The cardinality $|G_v|$ of $G_v$ is the \emph{degree} of $v$. 
The minimum degree amount the vertices of $G$ is denoted $\delta G$ and called \emph{minimum degree} of $G$, so
$
 \delta G = \min \{|G(v)|: v \in VG\}.
$
Let us call  \emph{strong density} to the ratio $\mu(G) = \delta G/(|G|-1)$.
In graph theory, the name \emph{density} is reserved for  $|EG|/\binom{|G|}{2}$, so, when the graph is regular, the strong and the ordinary density coincide.

Now, we will define a simple graph  from a weighted one.   Given a weighted graph $G$ and  a positive integer $k\geq \max w = \max_v w$ let us define the $k$-\emph{spinning} $S_k(G)$ of $G$ to be the graph with vertices $VG \times \Z_k$ and edges, 
$$
ES_k(G) = \{ (u,i)(v,j): uv \in EG,   j \in i+\N_{w_{uv}} \} \cup \{(u,i)(u,j): \forall i \neq j\},
$$
where  $\N_h = \{0,1,\ldots,h-1\}$, so $i+\N_{w_{uv}} =\{ i,i+1,\ldots, i+w_{uv}-1\} \}$. 
In Figure~\ref{Weighted and spinning graph} we show a weighted graph together  with its corresponding $2$-spinning graph.
\begin{figure}[htbp]
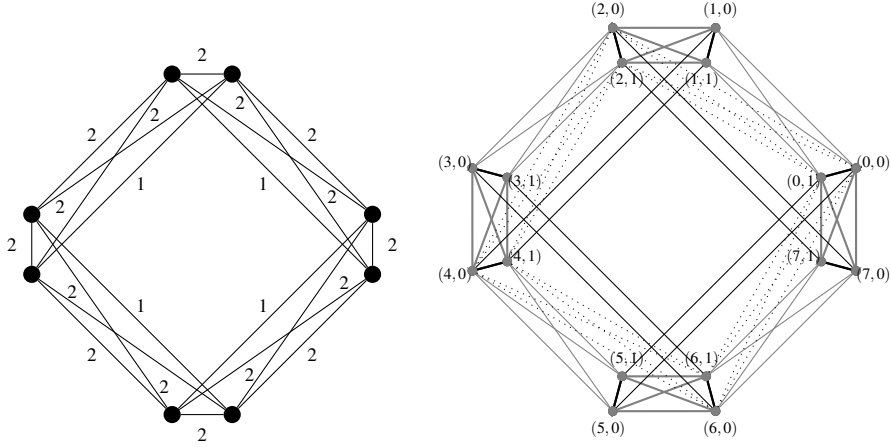

\begin{center}
\Gabcdconcrete{2}{2}{2}{1}
\Gabcdk{1}{1}{1}{0}{1}{180}{1.2}
\caption{A weighted graph on the left hand side with its $2$-spinning graph on the right hand side..}
\label{Weighted and spinning graph}
\end{center}
\end{figure}

Notice that the graph $G$ without weights is the quotient of $S_k(G)$ under the relation $(u,i)R(v,j) \iff u=v$, i.e. $G = S_k(G)/R$.
Besides, when all the weights are equal to $k$ we reobtain the Kronecker product of $G/R$ with the complete graph on $k$ vertices as described in previous works on this topic\cite{canale2015exotic,townsend2020dense}.

\subsection{Homogeneous Kuramoto Model}

Given a weighted graph $G$ with weight function $w$, 
the \emph{homogenous Kuramoto model of coupled oscillators  coupled through  $ G $}, is the system of differential equations  given by:
\begin{equation}\label{eq:Kuramoto}
\dot\theta_v = \sum_{u \in G_v} w_{uv} \sin(\theta_u - \theta_v) \qquad v \in VG.
\end{equation}

The system can be seen as running on the $n$-dimensional torus 
\( \mathbb{T}^n = (\R/2\pi\Z)^n \), which is a  compact manifold.
Besides, the system is  an analytic gradient one on $\mathbb{T}^n $, since  $\dot\theta = -\nabla U$,  for  $U$   the \emph{``potential energy''} defined as
$$
U(\theta_1,\ldots, \theta_n) = \frac12\sum_{uv\in EG} w_{uv}| e^{i\theta_u}-e^{i\theta_v}|^2
= w(G) - \sum_{uv \in EG} w_{uv}\cos(\theta_u  - \theta_v).
$$
Therefore, by the theory of analytic gradient systems\cite{Lojasiewicz}, every orbit goes to a consensus.
A point  $\theta \in \R^n$ is an  \emph{equilibrium} of  $G$ if it is an equilibrium of the system, i.e, if  \begin{equation}\label{eq:equilibrium}
\forall v \in V \qquad \sum_{u \in G_v} w_{uv}\sin(\theta_u - \theta_v) =0.
\end{equation}
Clearly $\theta   = (c,\ldots,c)$ are equilibrium and are called \emph{consensus}.

We say that a  graph \emph{synchronizes} iff  almost every solution tends to a consensus, i.e., if the set of initial conditions that give rise to orbits that do not tend to a consensus has  Lebesgue's measure zero. 
We known since the seminal work of Taylor\cite{taylor2012there}, that there is a limit $\mu_c<1$ called by Kassabov et al.  \emph{critical connectivity}\cite{Kassabov}, such that any graph with strong density greater than $\mu_c$ should synchronizes. We also known from the work of the latter authors\cite{Kassabov}, that the critical connectivity is at most 3/4.

Another important property of the system \eqref{eq:Kuramoto} is  the orthogonality of the solutions to  the vector $\one = (1,1,\ldots,1) $. Thus, the solutions are always running in a hyperplane orthogonal to $\one$, which is in fact a $(n-1)$--dimensional torus. So, we only care about the behavior inside these hyperplanes, though, for simplicity, the calculations will be done in $\R^n$. 

One way to see that a graph does not synchronizes is to prove that there exists a non consensus   equilibrium $\theta^ *$ where  the Hessian matrix  $U''_{\theta^*}$ of $U$ at $\theta^*$ has  all its  eigenvalue, but one,  positive.     Another way to express this condition is by 
 considering  the smaller eigenvalue $a(U''_{\theta^ *})$ in the hyperplane orthogonal to $\one$, which  is positive iff   all the eigenvalues are positive. Let us call  $a(U''_{\theta^*})$ the \emph{algebraic connectivity} of $\theta^*$ following the notation  in spectral graph theory where  \emph{algebraic connectivity} or also  \emph{Fiedler number} of the graph is reserved to the second smallest eigenvalue of the Laplacian matrix of the graph.  
The element $uv$ of $U''_{\theta^*}$ is
\begin{equation}\label{U''}
(U''_{\theta^*})_{uv} = \begin{cases}
- w_{uv}\cos(\theta^ *_u-\theta^ *_v) & u v\in EG, \\
\displaystyle\sum_{u\sim v} w_{uv}\cos(\theta^ *_u-\theta^ *_v) & u = v,\\
0 & \text{otherwise}.
\end{cases}
\end{equation}
Observe that, since $U''$ in a consensus coincides with the Laplacian matrix of the graph, then, the algebraic connectivity of a graph is the algebraic connectivity of any consensus.

\section{Equilibria of $k$-spinnings}
In this section we show a way to obtain equilibria of a spinning graph from equilibria of the original graph, as well as how the  former's linear stability  can be ensured by the stability of the last one.
Given a weighted graph $G$ and its $k$-spinning $S_k(G)$, if $\theta$ is an element of $\R^{|G|}$, let $\theta^\sigma$ be the element of $\R^{|G\times\Z_k|}$ given by
$$
\theta^\sigma_{(u,i)} = \theta_u \qquad \forall (u,i) \in G\times\Z_k.
$$

\begin{proposition}
A point  $\theta$ is an equilibrium of a weighted graph $G$  iff  the point $\theta^\sigma$ 
 is  an equilibrium of $S_k(G)$.
\end{proposition}
\begin{proof}
The result is a direct consequence of  the definition of equilibrium and the following equalities, where $w$ is the weight of $G$, given a vertex $(u,i)$ then
\begin{equation}\label{fundeq}
\sum_{(u,j) \in S_k(G)_{(v,i)}} \hspace{-5mm}\sin(\theta_{(u,i)}^\sigma - \theta_{(v,i)}^\sigma)=
\sum_{u \in G_v}\sum_{ j=1}^{w_{uv}-1} \sin(\theta_u - \theta_v)
=\sum_{u \in G_v} w_{uv}\sin(\theta^\sigma_u - \theta^\sigma_v) = 0.
\end{equation}
\end{proof}
Let us notice that this proposition can be extended to a much more general constructions.
For instance, if each vertex $v$ of the graph is substituted by a graph with order greater than the maximum degree of the vertices adjacent with $v$, the equations~\ref{fundeq}, will be still valid.
However, the  stability relationship between  both equilibria is weaker.

\begin{proposition}\label{prop}
If   $\theta$ is an equilibrium of a weighted graph $G$ with weight $w$ and Hessian matrix $U''_{\theta}$, then, the eigenvalues of  the Hessian matrix $U''_{\theta^\sigma}$ of $S_k(G)$  in $\theta^\sigma$ are   the  eigenvalues of   $U''_{\theta}$  with the same multiplicity plus a set of $(k-1)|G|$ (linearly independent) eigenvalues. These last eigenvalues are  positive for $k$ greater than twice the weight of the graph, i.e. $k > 2w(G)$.
In particular, the algebraic connectivity of $\theta$ in $G$ and  $\theta^\sigma$ in $S_k(G)$ have the same sign for $k$ large enough. 
\end{proposition}
\begin{proof}
From Eq.~\ref{U''}, the matrix $U''_{\theta^\sigma}$ has elements
$$
(U''_{\theta^\sigma})_{(u,i)(v,j)} = \begin{cases}
	k-1+\sum_{v \in G_u}w_{uv}\cos(\theta_u-\theta_v), & u=v, i = j,\\
	-1 &  u = v, i \neq j,\\
	-\cos(\theta^\sigma_{(u,i)}-\theta^\sigma_{(v,j)}) =-\cos(\theta_u-\theta_v)  & u \sim v, j \in i+\N_{w_uv},\\
	0 & \text{otherwise}.
\end{cases}
$$
Let $x$ be an eigenvector of $U''_{\theta^\sigma}$ with eigenvalue $\lambda$.
Therefore,
\begin{align*}
\lambda x_{(u,i)} &= (U''_{\theta^\sigma} x)_{(u,i)} =
\sum_{(v,j) \in S_k(G)_{(u,i)}} (U''_{\theta^\sigma})_{(u,j),(v,i)} x_{(v,i)}  \\
% =(U''_{\theta^\sigma})_{(u,j),(u,i)} x_{(u,i)}  + \sum_{(v,j) \sim (u,i)} (U''_{\theta^\sigma})_{(v,j),(u,i)} x_{(u,i)} 
&= \left(k-1+\sum_{v \in G_u} w_{uv}\cos(\theta_{u}-\theta_{v})\right) x_{(u,i)}
-\sum_{j\in \Z_k\setminus\{i\}}  x_{(v,j)}
-\sum_{(v,i) \in S_k(G)_{(u, j)}} \cos(\theta_{u}-\theta_{v}) x_{(v,j)}.
\end{align*}
Calling  $X_v$ to the sum  $\sum_{j\in \Z_k} x_{(v,j)}$, we have
\begin{align} \label{eqlambda}
\lambda x_{(u,i)} &= \left(k-1+(U''_\theta)_{uu}\right) x_{(u,i)}
-(X_u - x_{(u,i)})
-\sum_{v \in G_u} \cos(\theta_{u}-\theta_{v}) \sum_{j\in i+\N_{w_{uv}}}x_{(v,j)},\nonumber\\
&= \left(k+(U''_\theta)_{uu}\right) x_{(u,i)}
-X_u 
-\sum_{v \in G_u} \cos(\theta_{u}-\theta_{v}) \sum_{j\in i+\N_{w_{uv}}}x_{(v,j)}.
\end{align}
 If we sum these equations  for all $i\in \Z_k$   we obtain
 \begin{align*}
 \lambda X_u =
\left(k+(U''_\theta)_{uu}\right) X_u
-kX_u
-\sum_{v \in G_u}  \cos(\theta_u-\theta_v) w_{uv}X_v \\
\end{align*}
which after canceling the $(k-1)$'s, we obtain
 \begin{align*}
 \lambda X_u =
(U''_\theta)_{uu}X_u
-\sum_{v \in G_u} w_{uv} \cos(\theta_u-\theta_v) X_v.
\end{align*}
Therefore, $\lambda$ is an eigenvalue of $G$, with eigenvector $(X_u)_{u \in VG}$, unless $X_u=0$ for all $u$.
Conversely, let $(x_u)_{u \in VG}$ be an eigenvector of $U''_\theta$, with eigenvalue $\lambda$, then $x_{(u,i)}=x_u$ gives rise an eigenvalue of $U''_{\theta^\sigma}$ with eigenvalue $\lambda$.
Let us return to the case of $X_u=0$ for all $u$.  
Plugin $X_u=0$ into  Eq.~\ref{eqlambda} we get
\begin{align*}
\lambda x_{(u,i)} &=  \left(k+\sum_{v:uv \in EG} w_{uv}\cos(\theta_{u}-\theta_{v})\right) x_{(u,i)}
-\sum_{v:uv \in EG} \cos(\theta_{u}-\theta_{v}) \sum_{j\in i+\N_{w_{uv}}}x_{(v,j)}.
\end{align*}
Let $(u,i)$ be the greatest component of $x$ in absolute value, that we can suppose w.l.o.g. that it is 1, i.e. $x_{(u,i)}=1$ and $|x_{(v,j)}| \leq 1$ for all $(v,j)$. Then 
\begin{align*}
\lambda &=  k+\sum_{v\in G_u} w_{uv}\cos(\theta_{u}-\theta_{v})
-\sum_{v \in G_u} \cos(\theta_{u}-\theta_{v}) \sum_{j\in i+\N_{w_{uv}}}x_{(v,j)}\\
&=  k+\sum_{v \in G_u} \left(w_{uv}-\sum_{j\in i+\N_{w_{uv}}}x_{(v,j)}\right)\cos(\theta_{u}-\theta_{v}) > k -2w(G),
\end{align*}
which is positive if $k$ is large enough, at most twice the weight of $G$. Since the space of solutions of the system of equations $X_u=0$ for $u\in G$, has dimension $(k-1)|G|$, we conclude the proof.

\end{proof}

\section{The complexity of determining whether a simple graph  have a non-zero stable equilibrium. }

In 2015 R. Taylor  proved that \emph{determining whether weighted Kuramoto models have a non-zero linearly stable equilibrium is NP-hard}\cite{Taylor2015}. However, Taylor left open the unweighted case, though reducing it  to a problem he believed to  be NP-Hard.
 At the light of the previous sections, the $k$-spinning of the graph built by Taylor with $k$ large enough, will have a linearly stable equilibrium iff the graph has one, so determining if  the $k$-spinning has a stable equilibrium is as hard as determining if the graph does.Therefore we have
 \begin{theorem}
 Determining whether an homogeneous  Kuramoto model has a non consensus linearly stable equilibrium is NP-hard.
\end{theorem}

\section{A new  Lower Bound for the critical connectivity}

Let us consider the family of weighted graphs  $G_p$ with vertex set $\Z_8$ and edges 
 $EG_p = E_a \cup E_b \cup E_c \cup E_d $ where:
\begin{align*}
E_a =& \{(2i)(2i+1):\:  i \in \N_4\}, \\
E_b =& \{ i(i+2):\: i \in \Z_8\} ,\\
E_c =& \{ (2i)(2i+2):\: i \in \N_4 \},\\
E_d =& \{\{ (2i)(2i+4):\:  i \in \N_4\},
\end{align*}
with weight $w(uv) = \mathrm{t}$ if $uv \in E_\mathrm{t}$. Figure~\ref{Weighted and spinning graph} shows $G_{(2,2,2,1)}$.

\begin{figure}[htbp]
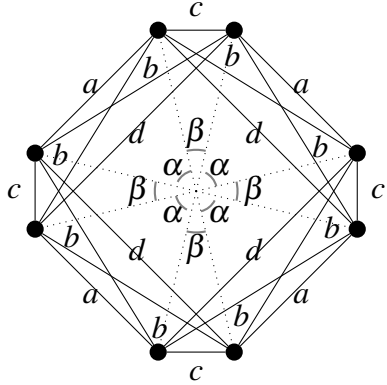

\begin{center}
\equilibrioalfa{1.1}{2}{13.2826}
\caption{The weighted graph $G_p$ with $p=(a,b,c,d)$.}
\label{Gabcd}
\end{center}
\end{figure}
Let us find a linearly stable equilibrium  $\theta^*$ of $G_p$ of the form 
\begin{equation*}
 \theta^*_{2i+s} = \frac{\pi}2 i+ s \frac{\beta}2  \qquad  \forall i \in \N_4, \quad s=\pm1.
\end{equation*}
If $\alpha =\pi/2-\beta$ the condition for $\theta^*$ to be an  equilibrium is:
$$
a\sin(\alpha) = c \sin(\beta) + d\sin(\alpha+2\beta ),
$$
since the edges of type  $b$ compensate each other, see Figure~\ref{Gabcd}.
Taking in count that   $\alpha + \beta = \pi/2$, then $\sin(\beta) = \cos(\alpha)$,  $\sin(\alpha+2\beta )= \sin(\alpha)$
and previous  equation  is equivalent to 
$
a\sin(\alpha) = c \cos(\alpha) + d\sin(\alpha),
$
i.e.  $(a-d)\tan \alpha = c $ so 
$
\alpha = \arctan(c/(a-c)).%\left(\frac{c}{a-d}\right).
$
\begin{lemma}\label{lema}
The equilibria $\theta$ of weighted graph $G_p$ with $p = (a,b,c,d )$ is  linearly stable equilibrium if $c^2 > 2ad$.
\end{lemma}
\begin{proof}
The matrix of $U''_{\theta}$ has elements:
\begin{equation*}
(U''_{\theta})_{uv}=
\begin{cases}
- w_{uv}\cos(\theta_u-\theta_v) & u\neq v, \\
\sum_{u\sim v} w_{uv}\cos(\theta_u-\theta_v) & u = v,\\
0& \text{ otherwise,}
\end{cases} 
\end{equation*}
thus
\begin{equation*}
(U''_{\theta})_{uv}=
\begin{cases}
- a\cos(\alpha) & u v \in E_a, \\
- b\cos(\alpha+\beta) & u v \in E_b, \\
- c\cos(\beta) & u v \in E_c, \\
- d\cos(\alpha+2\beta) & u v \in E_d, \\
a\cos(\alpha)+b\cos(\alpha+\beta)+c\cos(\beta)+d\cos(\alpha+2\beta) & u = v,\\
0& \text{ otherwise .}
\end{cases}\\
\end{equation*}
\begin{equation*}
(U''_{\theta})_{uv}=
\begin{cases}
- a\cos \alpha & u v \in E_a, \\
0  & u v \in E_b, \\
- c\sin \alpha & u v \in E_c, \\
d\cos \alpha & u v \in E_d, \\
a\cos\alpha+c\sin\alpha-d\cos\alpha & u = v,\\
0& \text{ otherwise .}
\end{cases}
\end{equation*}
Extracting $\cos \alpha$ as a factor, we obtain 
\begin{align*}
(U''_{\theta})_{uv}
&=\cos\alpha
\begin{cases}
- a & u v \in E_a, \\
- c\tan \alpha & u v \in E_c, \\
d & u v \in E_d, \\
a+c\tan\alpha-d & u = v,\\
0& \text{ otherwise .}
\end{cases}
=\cos\alpha
\begin{cases}
- a & u v \in E_a, \\
- \frac{c^2}{a-d} & u v \in E_c, \\
d & u v \in E_d, \\
a+\frac{c^2}{a-d}-d & u = v,\\
0& \text{ otherwise .}
\end{cases}
\end{align*}
Extracting $1/(a-d)$ as a factor, and setting $f=a-d$, we obtain 
\begin{align*}
(U''_{\theta})_{uv}
&=\frac{\cos\alpha}f
\begin{cases}
- af & u v \in E_a, \\
-c^2 & u v \in E_c, \\
df & u v \in E_d, \\
f^2+c^2 & u = v,\\
0& \text{ otherwise .}
\end{cases}
\end{align*}
i.e.
$$
\frac{f}{\cos\alpha}U''_{\theta}=\left( \begin {array}{cccccccc} 
c^2+f^2  &  -af  &  0  &  0  &  0  &  af-f^2  &  0  &  -c^2\\ 
-af  &  c^2+f^2  &  -c^2  &  0  &  af-f^2  &  0  &  0  &  0\\
0  &  -c^2  &  c^2+f^2  &  -af  &  0  &  0  &  0  &  af-f^2\\ 
0  &  0  &  -af  &  c^2+f^2  &  -c^2  &  0  &  af-f^2  &  0\\
 0  &  af-f^2  &  0  &  -c^2  &  c^2+f^2  &  -af  &  0  &  0\\ 
 af-f^2  &  0  &  0  &  0  &  -af  &  c^2+f^2  &  -c^2  &  0\\ 
 0  &  0  &  0  &  af-f^2  &  0  &  -c^2  &  c^2+f^2  &  -af\\ 
 -c^2  &  0  &  af-f^2  &  0  &  0  &  0  &  -af  &  c^2+f^2
\end {array} \right).
$$
The eigenvalues of this matrix are 0, $2f^2$, $2c^2$, $2c^2+2f^2$,
 $c^2+f^2+\sqrt{4a^2f^2-4af^3+c^4+f^4}$ with multiplicity two  and  $c^2+f^2-\sqrt{4a^2f^2-4af^3+c^4+f^4}$  with multiplicity two too (see Appendix~\ref{appendix}). Besides the null eigenvalue corresponding to the vector $\one$, the other eigenvalues  are positive but the last one, that will be positive iff  $2c^2f^2 > 4a^2f^2-4af^3$, i.e., if $c^2 > 2a^2-2af=2ad$, as we wanted to prove.
\end{proof}
We are now in conditions to find a lower bound for the critical connectivity by making $k$-spinnings of the graphs $G_p$ with a special choice of $p=(a,b,c,d)$.
%\begin{figure}[htbp]
%\begin{center}
%\begin{tabular}{ccccccc}
%
%\Gabcdk{0}{0}{0}{0}{1}{180}{1}&&
%\Gabcdk{1}{1}{1}{0}{1}{180}{1}&&
%\Gabcdk{0}{0}{0}{0}{2}{120}{1}\\
%%\Gabcdk{1}{1}{1}{1}{2}{120}{1}\\
%$p=(1,1,1,1)$&&
%$p=(2,2,2,1)$&&
%$p=(1,1,1,1)$
%%$p=(2,2,2,1,3)$
%\end{tabular}
%\caption{The $k$--spinning $S_k(G_p)$  of graphs  $G_p$ with $p=(a,b,c,d)$\label{Gabcdk}, with $k=2,2,3,$ respectively.}
%\label{default}
%\end{center}
%\end{figure}
Figure~\ref{Weighted and spinning graph} shows  the $2$-spinning of graph $G_p$ for  $p= (2,2,2,1)$.
We can see that the graphs $S_k(G_p)$  are  $\delta$-regular graphs  with $ \delta = a+2b+d+c+(k-1),$
and orders $|S_k(G_p)| = 8k$, so,  these graphs have strong density
$$
\mu(S_k(G_p)) = \frac{a+2b+d+c+(k-1)}{8k}.
$$
We will choose parameters $a,b,c,d$ proportional to $k$, so the strong density will tend to a constant.
Unfortunately, that choice does not verify the hypothesis of  Proposition~\ref{prop}, but a minor change in the arguments given in the proof of this proposition will be enough to  arrive at the same conclusion.
\begin{theorem}
The critical connectivity is  at least  11/16.
\end{theorem}
\begin{proof}
Consider the $k$-spinning graph $S_k(G_{p_k})$ of $G_{p_k}$ with $p_k=(k,k,k,\lceil k/2 \rceil-1)$, then 
$$
\mu(S_k(G_{p_k})) = \frac{4k+\lceil k/2 \rceil-1+k-1}{8k}  \to 11/16.
$$
Since  the choice of $a=b=c=k$ and $d=\lceil k/2 \rceil-1 $ verifies $c^2 =k^2 > 2k(\lceil k/2 \rceil-1)$, then, by the  Lemma ~\ref{lema} the equilibrium $\theta$ of the weighted graph $G_{p_k}$ is linearly stable.
We would like to apply Proposition~\ref{prop} to deduce that $G_{p_k}$ does not synchronize, but the hypothesis $k\geq w(G_{p_k})$ is not verified.
Nevertheless, we can prove that the $(k-1)|G|$ eigenvalues of $U''_{\theta^\sigma}$ that need this condition to be positive in the proof on Proposition~\ref{prop}, are still positive. 
Indeed, following the last part in  Proposition~\ref{prop}'s proof, let us consider the  formula for the eigenvalue $\lambda$:
\begin{align*}
\lambda &=  k+\sum_{v\in G_u} w_{uv}\cos(\theta_{u}-\theta_{v})
-\sum_{v\in G_u} \cos(\theta_{u}-\theta_{v}) \sum_{j\in i+\N_{w_{uv}}}x_{(v,j)}.
\end{align*}
In our case this formula turn out  to be
\begin{align*}
\lambda &=  k+  a\cos \alpha+c\sin \alpha-d\cos \alpha\\
&-\cos \alpha\sum_{j\in \Z_k} x_{(v,j)}
-\sin \alpha\sum_{j\in \Z_k} x_{(v',j)}
+\cos \alpha \sum_{j\in i+\N_d} x_{(v'',j)},
\end{align*}
where $uv \in E_a, uv' \in E_c, uv''\in E_d$. Then, since $a=c=k$ and $\sum_{j\in \Z_k} x_{(v,j)}=\sum_{j\in \Z_k} x_{(v',j)}=0$
\begin{align*}
\lambda &=  k+  k\cos \alpha+k\sin \alpha-d\cos \alpha 
+\cos \alpha \sum_{j\in i+\N_d} x_{(v'',j)}\\
&=  k+k\sin \alpha+\left(k-d+\sum_{j\in i+\N_d} x_{(v'',j)}\right)\cos \alpha,
\end{align*}
which is positive since  $|x_{(v'',j)}|\leq 1$ and $d=\lceil k/2\rceil-1$.
\end{proof}
It is worth to say that $a=b=c=k$ and $d=\lceil k/2\rceil-1$  maximize the strong density of the $k$-spinning graph  $G_{(a,b,c,d)}$ subject to the condition $c^2 > 2ad$, although it is not the only choice. 
We can also have chosen $d=b=c=k$ and $a=\lceil k/2\rceil-1$.

The first $k$ such that  the strong density $\mu_k$  is greater than the best lower bound known so far\cite{Yoneda}, i.e., $0.6838$,   is $k=34 $, that correspond to a graph with $272$ vertices, an order far away from exhaustive  numerical experiments.
The first $k$ such that  the strong density $\mu_k$  is greater than  than $0.6874$, is  $k=1250 $, that corresponds to a graph with  $10000$ vertices.

Although angles of $\pi/2$ in the equilibrium  contribute in other scenarios to   nonlinearity behavior\cite{townsend2020dense}, in our case, do not. In fact, the edges corresponding to that angle are those in $E_b$,  which  do not give rise to  constraints  of the equilibrium nor to the positive definition of the quadratic forms,  because $\cos \pi/2 =0$, so $b=k$ is always an optimal choice.  
However,

\section{Consequences }

We have develop a way to associate an unweighted graph from a weighted one such that the existence of linearly stable equilibria in the last one implies, under some conditions, the existence of such equilibria in the former.
This technique allows us to improve the known lower bound for the critical connectivity from 0.6838 to 0.6875.
We hope this technique will allow to  improve that limit to at least 0.7495, according to the intuitive  arguments given by Townsend et al.\cite{townsend2020dense}.
If these arguments are true, the gap from 0.7495 to 0.75 will require a totally different approach.
On the other hand we also were able to prove that deciding if a graph has non trivial linearly stable equilibria is an NP-Hard problem based on a result by Taylor\cite{Taylor2015} and the weighted-to-unweighted technique.
However the complexity of deciding if a graph synchronizes is still open with only partial results\cite{canale2010}, moreover, we do not even know if that problem is computable.
The last question is related to the absences  of a combinatorial characterization of synchronizing graphs.
This issue is  similar to the characterization of planar graphs, in the sense that both types of graphs, i.e., planar and synchronizing are defined in a non combinatorial way. Kazimierz Kuratowski gives a combinatorial characterization of the formers  in terms of forbidden induced subgraphs. Unfortunately, this particular way is not possible for synchronizing graphs, since for any fixed graph there are synchronized and unsynchronized ones with that graph as an induced subgraph\cite{Canale2008}.
The biggest step into a combinatorial characterization was  done by Taylor\cite{taylor2012there} when he found a connection between the minimum degree and the synchronization property, however we are still far away from a final
solution.

\section{Acknowledges}

The author thanks Alex Townsend for a useful conversation. 

\appendix
\section{Eigenvector of matrix in proof of Lemma~\ref{lema}\label{appendix}}
In this section we provide a base of eigenvectors of the scaled Hessian  matrix in proof of Lemma~\ref{lema}, i.e.
$$M=\left( \begin {array}{cccccccc} 
c^2+f^2  &  -af  &  0  &  0  &  0  &  af-f^2  &  0  &  -c^2\\ 
-af  &  c^2+f^2  &  -c^2  &  0  &  af-f^2  &  0  &  0  &  0\\
0  &  -c^2  &  c^2+f^2  &  -af  &  0  &  0  &  0  &  af-f^2\\ 
0  &  0  &  -af  &  c^2+f^2  &  -c^2  &  0  &  af-f^2  &  0\\
 0  &  af-f^2  &  0  &  -c^2  &  c^2+f^2  &  -af  &  0  &  0\\ 
 af-f^2  &  0  &  0  &  0  &  -af  &  c^2+f^2  &  -c^2  &  0\\ 
 0  &  0  &  0  &  af-f^2  &  0  &  -c^2  &  c^2+f^2  &  -af\\ 
 -c^2  &  0  &  af-f^2  &  0  &  0  &  0  &  -af  &  c^2+f^2
\end {array} \right).
$$
A basis of eigenvector of $M$ are the column vectors of  the following matrix,
$$V=\left( \begin {array}{cccccccc} -{\frac {r}{{c}^{2}}}&{\frac {-2\,af+{f}^{2}}{{c}^{2}}}
&{\frac {r}{{c}^{2
}}}&{\frac {-2\,af+{f}^{2}}{{c}^{2}}}&1&1&-1&-1\\ \noalign{\medskip}{
\frac {f \left( 2\,a-f \right) }{{c}^{2}}}&{\frac {r}{{c}^{2}}}&{\frac {f \left( 2\,a-f
 \right) }{{c}^{2}}}&-{\frac {r}{{c}^{2}}}&1&-1&-1&1\\ \noalign{\medskip}0&-1&0&-1&1&-1&
1&-1\\ \noalign{\medskip}-1&0&-1&0&1&1&1&1\\ \noalign{\medskip}{\frac 
{r}{{c}^{2}}}&{
\frac {f \left( 2\,a-f \right) }{{c}^{2}}}&-{\frac {r}{{c}^{2}}}&{\frac {f \left( 2\,a-f
 \right) }{{c}^{2}}}&1&1&-1&-1\\ \noalign{\medskip}{\frac {-2\,af+{f}^
{2}}{{c}^{2}}}&-{\frac {r}{{c}^{2}}}&{\frac {-2\,af+{f}^{2}}{{c}^{2}}}&{\frac {r}{{c}^{2}}}&1&-1&-1&1
\\ \noalign{\medskip}0&1&0&1&1&-1&1&-1\\ \noalign{\medskip}1&0&1&0&1&1
&1&1\end {array} \right)
$$
with $r = \sqrt {4\,{a}^{2}{f}^{2}-4\,a{f}^{3}+{c}^{4}+{f}^{4}}$, as can be checked.
Therefore $M\cdot V$  is a diagonal matrix with the corresponding eigenvalues of the eigenvector. 
More concretely,
$$
M\cdot V =  \left( \begin {array}{ccccccccc} 
{c}^{2}+{f}^{2}+r &0&0&0&0&0&0&0\\
0&{c}^{2}+{f}^{2}+r&0&0&0&0&0&0\\ 
0&0&{c}^{2}+{f}^{2}-r&0&0&0&0&0\\
0&0&0&{c}^{2}+{f}^{2}-r&0&0&0&0\\ 
0&0&0&0&0&0&0&0\\ 
0&0&0&0&0&2\,{f}^{2}&0&0\\ 
0&0&0&0&0&0&2\,{c}^{2}&0\\ 
0&0&0&0&0&0&0&2\,{c}^{2}+2\,{f}^{2}
\end {array} \right)
$$

\bibliography{BestLowerBound}

\end{document}